\documentclass[12pt,a4paper]{amsart}

\usepackage{amsmath}
\usepackage{amssymb}
\usepackage{amsfonts,amsthm,mathdots,url,tabu}
\usepackage{cite}
\usepackage{pdflscape}

\newcommand{\M}{\mathcal{M}}
\newcommand{\QM}{\mathcal{QM}}
\newcommand{\QS}{\mathcal{QS}}
\newcommand{\QE}{\mathcal{QE}}
\newcommand{\QEz}{\mathcal{Q}\hspace*{-1pt}^{^0}\hspace*{-3pt}\mathcal{E}}

\renewcommand{\S}{\mathcal{S}}
\newcommand{\C}{\mathbb{C}}

\newcommand{\id}{\mathrm{id}}

\newcommand{\BigOh}{\mathcal{O}}

\newtheorem{prop}{Proposition}
\newtheorem{theorem}{Theorem}

\theoremstyle{definition}

\newtheorem*{acknowledgement}{Acknowledgement}
\begin{document}
\title[Extremal quasimodular forms]%
{Asymptotic expansions for the coefficients of extremal quasimodular forms and
  a conjecture of Kaneko and Koike}
\author{Peter J. Grabner}
\thanks{The author is supported by the Austrian Science Fund FWF project F5503
  (part of the Special Research Program (SFB) ``Quasi-Monte Carlo Methods:
  Theory and Applications'')}
\address{Institut f\"ur Analysis und Zahlentheorie,
Technische Universit\"at Graz,
Kopernikusgasse 24,
8010 Graz,
Austria}
\email{peter.grabner@tugraz.at}
\begin{abstract}
  Extremal quasimodular forms have been introduced by M.~Kaneko and M.~Koike
  as quasimodular forms which have maximal possible order of vanishing at
  $i\infty$.  We show an asymptotic formula for the Fourier coefficients of
  such forms. This formula is then used to show that all but finitely many
  Fourier coefficients of such forms of depth $\leq4$ are positive, which
  partially solves a conjecture stated by M.~Kaneko and M.~Koike. Numerical
  experiments based on constructive estimates confirm the conjecture for
  weights $\leq200$ and depths between $1$ and $4$.
\end{abstract}
\maketitle
\section{Introduction}\label{sec:introduction}
``Quasimodular forms'' as a notion were introduced by M.~Kaneko and D.~Zagier
in \cite{Kaneko_Zagier1995:generalized_jacobi_theta}. They have found
applications in various areas of mathematics and are of interest on their own.
For excellent introductions to the subject we refer to
\cite{Royer2012:quasimodular_forms_introduction,
  Zagier2008:elliptic_modular_forms,Choie_Lee2019:jacobi_like_forms}.

In \cite{Kaneko_Koike2006:extremal_quasimodular_forms} M.~Kaneko and M.~Koike
introduced the notion of \emph{extremal quasimodular forms}. These are
quasimodular forms of weight $w$ and depth $r$, which show extremal order of
vanishing at $z=i\infty$ amongst all forms of that weight and depth. The
authors conjectured certain arithmetic properties of the Fourier coefficients
of these forms for depth $r\leq4$. These were established in
\cite{Grabner2020:quasimodular_forms}; in
\cite{Pellarin_Nebe2020:extremal_quasi_modular} and
\cite{Mono2020:conjecture_kaneko_koike} these were proved independently for
$r=1$. A second part of the conjecture stated in
\cite{Kaneko_Koike2006:extremal_quasimodular_forms} concerned the positivity of
the Fourier coefficients of extremal quasimodular forms. In this paper we prove
that for any $w$ and $r\leq4$ all but possibly finitely many Fourier
coefficients are positive. Using a bound given by P.~Jenkins and R.~Rouse
\cite{Jenkins_Rouse2011:bounds_coefficients_cusp} we could verify the
conjecture for $1\leq r\leq4$ and $w\leq200$.

\section{Notation and preliminary results}
\label{sec:notat-prel-results}

In this section we collect some basic facts about modular and quasimodular
forms.
\subsection{Modular forms}\label{sec:modular-forms}
The modular group $\Gamma$ is the group of $2\times2$-matrices with integer
entries and determinant $1$
\begin{equation*}
  \Gamma=\mathrm{PSL}(2,\mathbb{Z})=\left\{
    \begin{pmatrix}
      a&b\\c&d
    \end{pmatrix}\Bigm| a,b,c,d\in\mathbb{Z}, ac-bd=1
  \right\}/\{\pm I\}.
\end{equation*}
It acts on the upper
half plane $\mathbb{H}=\{z\in\mathbb{C}\mid\Im z>0\}$ by M\"obius
transformation
\begin{equation*}
  \begin{pmatrix}
      a&b\\c&d
    \end{pmatrix}z=\frac{az+b}{cz+d}.
\end{equation*}
The group $\Gamma$ is generated by
\begin{equation}
  \label{eq:ST}
  Sz=-\frac1z\quad Tz=z+1,
\end{equation}
which satisfy the relations $S^2=\id$ and $(ST)^3=\id$.  A holomorphic function
$f:\mathbb{H}\to\mathbb{C}$ is called a \emph{holomorphic modular form
  of weight $w$}, if it satisfies
\begin{equation}
  \label{eq:modular}
  (cz+d)^{-w}f\left(\frac{az+b}{cz+d}\right)=f(z)
\end{equation}
for all $z\in\mathbb{H}$ and all $\bigl( \begin{smallmatrix}a & b\\ c &
  d\end{smallmatrix}\bigr)\in\Gamma$, and the limit
\begin{equation*}
  f(i\infty):=\lim_{\Im z\to+\infty}f(z)
\end{equation*}
exists.  The vector space $\M_w(\Gamma)$ of holomorphic modular forms is
non-trivial only for even $w\geq4$ and $w=0$. Its dimension equals
\begin{equation*}
  \dim\M_w(\Gamma)=
  \begin{cases}
    \left\lfloor\frac w{12}\right\rfloor&\text{for }w\equiv 2\pmod{12}\\
    \left\lfloor\frac w{12}\right\rfloor+1&\text{otherwise.}
  \end{cases}
\end{equation*}
Prominent examples of modular forms are the Eisenstein series
\begin{equation}
  \label{eq:eisenstein-2k}
  E_{2k}(z)=\frac1{2\zeta(2k)}\sum\limits_{(m,n)\in\mathbb{Z}\setminus\{(0,0)\}}
  \frac1{(mz+n)^{2k}}
\end{equation}
for $k\geq2$, which are modular forms of weight $2k$. They admit a Fourier
expansion (setting $q=e^{2\pi iz}$ as usual in this context)
\begin{equation}\label{eq:eisenstein-fourier}
  E_{2k}=1-\frac{4k}{B_{2k}}\sum_{n=1}^\infty\sigma_{2k-1}(n)q^n,
\end{equation}
where $\sigma_\alpha(n)=\sum_{d\mid n}d^\alpha$ denotes the divisor sum of
order $\alpha$ and $B_{2k}$ denote the Bernoulli numbers. The defining series
\eqref{eq:eisenstein-2k} does not converge for $k=1$ in the given
form. Nevertheless, the series \eqref{eq:eisenstein-fourier} converges for
$k\geq1$.  This entails a slightly more complicated transformation behaviour
under the action of $S$
\begin{equation}
  \label{eq:E2S}
  z^{-2}E_2(Sz)=E_2(z)+\frac6{\pi iz}.
\end{equation}

Every holomorphic modular form can be expressed as a complex polynomial in
$E_4$ and $E_6$, furthermore
\begin{equation*}
  \bigoplus_{k=0}^\infty \M_{2k}(\Gamma)=\C[E_4,E_6].
\end{equation*}
By the invariance under $T$, every holomorphic modular form $f$ has a Fourier
expansion
\begin{equation*}
  f(z)=\sum_{n=0}^\infty a_f(n)e^{2\pi inz}=\sum_{n=0}^\infty a_f(n)q^n.
\end{equation*}
In the sequel we will follow the convention to freely switch between dependence
on $z$ and $q$.

A holomorphic form $f$ is called a \emph{cusp form}, if $f(i\infty)=0$. The
prototypical example of a cusp form is
\begin{equation}
  \label{eq:Delta}
  \Delta=\frac1{1728}\left(E_4^3-E_6^2\right).
\end{equation}
The space of cusp forms is denoted by $\mathcal{S}_w(\Gamma)$. Since we only
deal with modular forms for the full modular group $\Gamma$, we will omit
reference to the group in the sequel.

For a detailed introduction to the theory of modular forms we refer to
\cite{Shimura2012:modular_forms_basics,
  Berndt_Knopp2008:heckes_theory_modular,
  Bruinier_Geer_Harder+2008:1_2_3-modular,
  Stein2007:modular_forms_computational,
  Diamond_Shurman2005:first_course_modular,
  Iwaniec1997:topics_classical_automorphic,
  Lang1995:introduction_to_modular}.

\subsection{Quasimodular forms}\label{sec:quasimodular-forms}
The vector space of quasimodular forms of weight $w$ and depth $\leq r$ is
given by
\begin{equation}
  \label{eq:quasi-space}
  \QM_w^{r}=\bigoplus_{\ell=0}^rE_2^\ell\M_{w-2\ell}.
\end{equation}
Quasimodular forms occur naturally as derivatives of modular forms (see
\cite{Royer2012:quasimodular_forms_introduction,
  Zagier2008:elliptic_modular_forms,Choie_Lee2019:jacobi_like_forms}). This
aspect will be used and elaborated later.

Throughout this paper we use the notation
\begin{equation*}
  Df=\frac1{2\pi i}\frac{df}{dz}=q\frac{df}{dq}.
\end{equation*}
Higher derivatives are always expressed as powers of $D$. Upper
indices will never denote derivatives.

With this notation Ramanujan's identities read
\begin{equation}
  \label{eq:ramanujan}
  \begin{split}
    DE_2&=\frac1{12}\left(E_2^2-E_4\right)\\
    DE_4&=\frac13\left(E_2E_4-E_6\right)\\
    DE_6&=\frac12\left(E_2E_6-E_4^2\right)\\
    D\Delta&=E_2\Delta.
  \end{split}
\end{equation}
These give rise to the definition of the Ramanujan-Serre derivative
\begin{equation*}
  \partial_wf=Df-\frac w{12}E_2f,
\end{equation*}
where $w$ is (related to) the weight of $f$. We will use the product rule
\begin{equation*}
  \partial_{w_1+w_2}(fg)=\left(\partial_{w_1}f\right)g+
  f\left(\partial_{w_2}g\right)
\end{equation*}
and also make frequent use of the following immediate consequences of
\eqref{eq:ramanujan}
\begin{equation}
  \label{eq:serre-ramanujan}
  \begin{split}
    \partial_1E_2&=-\frac1{12}E_4\\
    \partial_4E_4&=-\frac13E_6\\
    \partial_6E_6&=-\frac12E_4^2\\
    \partial_{12}\Delta&=0.
  \end{split}
\end{equation}
From the second and third equation together with the fact that every
holomorphic modular form is a polynomial in $E_4$ and $E_6$, it follows
immediately that for a form $f\in\M_w$ we have $\partial_wf\in\M_{w+2}$, and
for $f\in\mathcal{S}_w$ we have $\partial_wf\in\mathcal{S}_{w+2}$.

We set
\begin{equation*}
  \QS_w^r=\bigoplus_{\ell=0}^rE_2^\ell\S_{w-2\ell}=\Delta\QM_{w-12}^r
\end{equation*}
for the space of quasimodular forms with cusp form coefficients for all powers
of $E_2$. For the spaces of quasimodular forms we have the alternative
descriptions as
\begin{equation}\label{eq:QM-D}
  \QM_w^r=\bigoplus_{\ell=0}^rD^\ell\M_{w-2\ell},
\end{equation}
see for instance \cite[Proposition~14.3]{Choie_Lee2019:jacobi_like_forms}, and
\begin{equation}\label{eq:QS-D}
  \QS_w^r=\bigoplus_{\ell=0}^rD^\ell\S_{w-2\ell}.
\end{equation}
The second decomposition follows from the last equation in
\eqref{eq:ramanujan}.  The direct sum in \eqref{eq:QM-D} can be further refined
as
\begin{equation}\label{eq:QM-decomp}
  \QM_w^r=\bigoplus_{\ell=0}^rD^\ell\left(\S_{w-2\ell}\oplus\C
    E_{w-2\ell}\right)=
  \QS_w^r\oplus\bigoplus_{\ell=0}^r\C D^\ell E_{w-2\ell}.
\end{equation}

We set
\begin{equation}
  \label{eq:eisenstein}
  \QE_w^r=\QM_w^r/\QS_w^r
\end{equation}
the ``Eisenstein space''. We write $\overline{f}$ for $f+\QS_w^r$. Notice that
this notation implicitly depends on $w$ and $r$. Then $D$ maps $\QE_w^r$ to
$\QE_{w+2}^{r+1}$. Similarly, $\partial_{w-r}$ maps $\QE_{w}^r$ to
$\QE_{w+2}^r$ by \cite[Lemma~2.2]{Grabner2020:quasimodular_forms}. Notice that
it makes sense to define $\overline{f}(i\infty)$ for
$\overline{f}\in\QE_w^r$. Thus we can define
\begin{equation*}
  \QEz_w^r=\{\overline{f}\in\QE_w^r\mid \overline{f}(i\infty)=0\}.
\end{equation*}
 Then for $w\geq2r+4$
\begin{equation*}
  \{\overline{D^\ell E_{w-2\ell}}\mid \ell=0,\ldots,r\}
\end{equation*}
is a basis of $\QE_w^r$, and
\begin{equation*}
  \{\overline{D^\ell E_{w-2\ell}}\mid \ell=1,\ldots,r\}
\end{equation*}
is a basis of $\QEz_w^r$.

Notice that for $v,w\geq4$
\begin{equation*}
  E_{v+w}-E_vE_w
\end{equation*}
is a cusp form of weight $v+w$, which allows to write
\begin{equation}\label{eq:EwEv}
  \overline{E_vE_w}=\overline{E_{v+w}}.
\end{equation}

Using the definition of the Serre derivative, we obtain
\begin{equation*}
  \partial_wE_w=DE_w-\frac w{12}E_2E_w,
\end{equation*}
which is a modular form of weight $w+2$ with constant coefficient $-\frac
w{12}$. Thus we have
\begin{equation*}
  \partial_wE_w=-\frac w{12}E_{w+2}+\text{cusp form},
\end{equation*}
which we write as
\begin{equation}
  \label{eq:serre-Ew}
  \overline{\partial_wE_w}=\partial_w\overline{E_w}=
  -\frac w{12}\overline{E_{w+2}}.
\end{equation}
Using this we obtain
\begin{equation*}
  \partial_{w-r}\overline{E_{w-2\ell}E_2^\ell}=
  \overline{\left(\partial_{w-2\ell}E_{w-2\ell}\right)E_2^\ell}+
  \overline{E_{w-2\ell}\left(\partial_{2\ell-r}E_2^\ell\right)},
\end{equation*}
from which we derive
\begin{equation}
  \label{eq:serre-EwE2}
  \begin{split}
    \partial_{w-r}\overline{E_{w-2\ell}E_2^\ell}&=
    -\frac\ell{12}\overline{E_{w-2\ell}E_4E_2^{\ell-1}}
    -\frac{w-2\ell}{12}\overline{E_{w-2\ell+2}E_2^\ell}-
    \frac{\ell-r}{12}\overline{E_{w-2\ell}E_2^{\ell+1}}\\
    &=-\frac\ell{12}\overline{E_{w-2\ell+4}E_2^{\ell-1}}
    -\frac{w-2\ell}{12}\overline{E_{w-2\ell+2}E_2^\ell}-
    \frac{\ell-r}{12}\overline{E_{w-2\ell}E_2^{\ell+1}}.
  \end{split}
\end{equation}
Here we have used
\begin{equation*}
  \partial_{2\ell-r}E_2^\ell=-\frac\ell{12}E_4E_2^{\ell-1}
  -\frac{\ell-r}{12}E_2^{\ell+1}.
\end{equation*}

Consider the forms
\begin{equation}\label{eq:f_wk}
  f_w^{(k)}=\sum_{\ell=0}^k(-1)^\ell\binom k\ell E_{w-2\ell}E_2^\ell,
\end{equation}
where we set $E_0=1$ and omit the term for $w-2\ell=2$, which only occurs if
$w\leq2k+2$.
We compute using \eqref{eq:serre-EwE2}
\begin{align*}
  &\partial_{w-r}\overline{f_w^{(k)}}=\sum_{\ell=0}^k(-1)^\ell\binom k\ell\\
  &\times
  \left(-\frac\ell{12}\overline{E_{w-2\ell+4}E_2^{\ell-1}}
    -\frac{w-2\ell}{12}\overline{E_{w-2\ell+2}E_2^\ell}
    -\frac{\ell-r}{12}\overline{E_{w-2\ell}E_2^{\ell+1}}\right).
\end{align*}
Expanding the binomial coefficients and shifting the summation index gives
\begin{align*}
  &\partial_{w-r}\overline{f_w^{(k)}}=-\frac1{12}
  \sum_{\ell=0}^{k+1}(-1)^\ell \overline{E_{w-2\ell+2}E_2^\ell}\\
  &\times\left(-k\binom{k-1}\ell+w\binom k\ell
    -2k\binom{k-1}{\ell-1}-k\binom{k-1}{\ell-2}+r\binom k{\ell-1}\right),
\end{align*}
where we have set $\binom km=0$ for $m<0$ and $m>k$. The term in parenthesis is
then equal to
\begin{equation*}
  (w-r)\binom k\ell+(r-k)\binom{k+1}\ell,
\end{equation*}
which gives
\begin{equation}\label{eq:Serre-fk}
  \partial_{w-r}\overline{f_w^{(k)}}=-\frac{w-r}{12}\overline{f_{w+2}^{(k)}}
  -\frac{r-k}{12}\overline{f_{w+2}^{(k+1)}}
\end{equation}
for $k=0,\ldots,r$ and $w\geq2r+2$ (the case $k=r$ and $w=2r+2$ has to be
checked separately).

We also have for $w\geq4$
\begin{equation}\label{eq:DEw}
  \overline{D^kE_w}=(-1)^k\frac{(w)_k}{12^k}\overline{f_{w+2k}^{(k)}},
\end{equation}
where $(w)_k=w(w+1)\cdots(w+k-1)$ denotes the Pochhammer symbol. We prove
\eqref{eq:DEw} by induction. For $k=0$ it obviously holds. The induction step
reads as (using \eqref{eq:Serre-fk} for $r=k$) 
\begin{align*}
  \overline{D^{k+1}E_w}&=(-1)^k\frac{(w)_k}{12^k}\overline{Df_{w+2k}^{(k)}}\\
  &=  (-1)^k\frac{(w)_k}{12^k}\left(\overline{\partial_{w+k}f_{w+2k}^{(k)}}
    +\frac{w+k}{12}\overline{E_2f_{w+2k}^{(k)}}\right)\\
  &=  (-1)^{k+1}\frac{(w)_k}{12^{k+1}}\left((w+k)\overline{f_{w+2k+2}^{(k)}}
    -(w+k)\overline{E_2f_{w+2k}^{(k)}}\right)\\
  &=(-1)^{k+1}\frac{(w)_{k+1}}{12^{k+1}}\overline{f_{w+2k+2}^{(k+1)}}.
\end{align*}
Furthermore, we have
\begin{equation}
  \label{eq:f_w+l}
  \overline{E_{v}f_w^{(k)}}=\overline{f_{w+v}^{(k)}}
\end{equation}
for $w\geq2k+4$ and $v\geq4$, which follows from \eqref{eq:EwEv}. For later
reference we notice that
\begin{equation}\label{eq:D-eisenstein}
  D^kE_w(z)=-\frac{2w}{B_w}\sum_{n=1}^\infty n^k\sigma_{w-1}(n)q^n
\end{equation}
for $k\geq1$ and even $w\geq2$.

We will use the following convention for iterated Serre derivatives 
\begin{equation*}
  \partial_w^0f=f,\quad
  \partial_w^{k+1}=\partial_{w+2k}\left(\partial_w^kf\right).
\end{equation*}
With this we get the following expressions for higher derivatives in terms of
Serre derivatives, which we will need later
\begin{equation}
  \label{eq:D-serre}
  \begin{split}
    Df&=\partial_{w-2}f+E_2\frac{w-2}{12}f\\
    D^2f&=\left(\partial_{w-4}^2f-\frac{w-4}{144}E_4f\right)
    +E_2\frac{w-3}6\partial_{w-4}f
    +E_2^2\frac{(w-3)(w-4)}{144}f\\
    D^3f&=\left(\partial_{w-6}^3f-\frac{3w-16}{144}E_4\partial_{w-6}f+
      \frac{w-6}{432}E_6f\right)\\
    &+E_2\left(\frac{w-4}4\partial_{w-6}^2f-\frac{(w-4)(w-6)}{576}E_4f\right)\\
    &+E_2^2\frac{(w-4)(w-5)}{48}\partial_{w-6}f+
    E_2^3\frac{(w-4)(w-5)(w-6)}{1728}f\\
    D^4f&=\left(\partial_{w-8}^4f-\frac{3w-20}{72}E_4\partial_{w-8}^2f
      +\frac{2w-15}{216}E_6\partial_{w-8}f+\frac{(w-8)(w-14)}{6192}E_4^2f
    \right)\\
    &+E_2\left(\frac{w-5}3\partial_{w-8}^3f
      -\frac{(3w-22)(w-5)}{432}E_4\partial_{w-8}f
      +\frac{(w-5)(w-8)}{1296}E_6f\right)\\
    &+E_2^2\left(\frac{(w-5)(w-6)}{24}\partial_{w-8}^2f
      -\frac{(w-5)(w-6)(w-8)}{3456}E_4f\right)\\
    &+E_2^3\frac{(w-5)(w-6)(w-7)}{432}\partial_{w-8}f
    +E_2^4\frac{(w-5)(w-6)(w-7)(w-8)}{20736}f.
  \end{split}
\end{equation}

\begin{prop}\label{prop1}
  Let $g\in\QS_w^r$ be given by its Fourier expansion
  \begin{equation*}
    g(z)=\sum_{n=1}^\infty a(n)q^n.
  \end{equation*}
  Then $a(n)=\BigOh(n^{\frac{w-1}2}\sigma_0(n))$.
\end{prop}
\begin{proof}
  Let $g$ first be in $D^\ell\S_{w-2\ell}$ for some $\ell\geq0$. Then $g$ is
  the $\ell$-th derivative of a cusp form $G\in\S_{w-2\ell}$. By Deligne's
  estimate \cite[Th\'eor\`eme~8.2]{Deligne1974:la_conjecture_weil} (see also
  \cite[Section~14.9]{Iwaniec_Kowalski2004:analytic_number_theory}) the Fourier
  coefficients of $G$ are bounded by $\BigOh(n^{\frac{w-1}2-\ell}\sigma_0(n))$
  . The effect of $\ell$-fold derivation is multiplication with $n^\ell$, which
  gives the desired estimate for this special case. Since the same estimate
  holds for all spaces in the direct sum \eqref{eq:QS-D}, it holds for every
  $g$ in $\QS_w^r$.
\end{proof}
\begin{prop}\label{prop2}
  Let $f_w^{(k)}$ be the form given by \eqref{eq:f_wk} and let
  \begin{equation*}
    f_w^{(k)}=\delta_{k,0}+\sum_{n=1}^\infty a_w^{(k)}(n)q^n
  \end{equation*}
  be its Fourier expansion. Then the asymptotic expansion
  \begin{equation}
    \label{eq:a_wk}
    a_w^{(k)}(n)=
    \begin{cases}
      -\frac{2w}{B_w}\sigma_{w-1}(n)&\text{for }k=0\\
      (-1)^{k+1}\frac{2\cdot12^k}{(w-2k+1)_{k-1}B_{w-2k}}n^k\sigma_{w-2k-1}(n)
      &\text{for }k>0\\
      \quad+\BigOh\left(n^{\frac{w-1}2}\sigma_0(n)\right)&
    \end{cases}
  \end{equation}
  holds.
\end{prop}
\begin{proof}
  The case $k=0$ is just the Fourier expansion of the Eisenstein series given
  in \eqref{eq:eisenstein-2k}. For $k>0$ we use \eqref{eq:DEw} to obtain
  \begin{equation*}
    f_w^{(k)}=(-1)^k\frac{12^k}{(w-2k)_k}D^kE_{w-2k}+h_{w}^{(k)},
  \end{equation*}
  where $h_{w}^{(k)}\in\QS_{w}^{(k)}$. The Fourier coefficient of the
  first term equals
  \begin{equation*}
    (-1)^{k+1}\frac{12^k}{(w-2k)_k}\frac{2(w-2k)}{B_{w-2k}}n^k\sigma_{w-2k-1}(n)
  \end{equation*}
  using \eqref{eq:D-eisenstein}.  The Fourier coefficient of $h_{w}^{(k)}$ is
  estimated using Proposition~\ref{prop1} to obtain \eqref{eq:a_wk}.
\end{proof}

We recall the dimension formulas for the spaces $\QM_w^r$ for
$1\leq r\leq4$:
\begin{equation}
  \label{eq:dimqm}
  \begin{split}
  \dim\QM_w^1&=\left\lfloor\frac w{6}\right\rfloor+1\\
  \dim\QM_w^2&=\left\lfloor\frac w{4}\right\rfloor+1\\
  \dim\QM_w^3&=\left\lfloor\frac w{3}\right\rfloor+1\\
  \dim\QM_w^4&=
  \begin{cases}
    \left\lfloor\frac {5w}{12}\right\rfloor&\text{if }w\equiv10\pmod{12}\\
    \left\lfloor\frac {5w}{12}\right\rfloor+1&\text{otherwise;}
  \end{cases}
  \end{split}
\end{equation}
see, for instance \cite[Proposition~2.1]{Grabner2020:quasimodular_forms}.
\section{Extremal quasimodular forms}
\label{sec:extr-quas-forms}
The notion of an \emph{extremal quasimodular form} was introduced in
\cite{Kaneko_Koike2006:extremal_quasimodular_forms}. They are defined as
quasimodular forms achieving the maximal possible order of vanishing at
$z=i\infty$ for given weight $w$ and depth $r$. It follows from a simple
dimension argument that the order of vanishing $\dim\QM_w^r-1$ can be
achieved. It was shown in
\cite[Theorem~1.3]{Pellarin_Nebe2020:extremal_quasi_modular} and independently
in \cite[Remark~4.7]{Grabner2020:quasimodular_forms} that for
$r\leq4$ this is actually the precise order of vanishing for such forms.

In \cite{Kaneko_Koike2006:extremal_quasimodular_forms} differential equations
satisfied by extremal quasimodular forms are found for $r=1$ and
$r=2$. Furthermore, two conjectures about these forms for $r\leq4$ are stated:
\begin{itemize}
\item if the first non-zero Fourier coefficient of the extremal quasimodular
  form equals $1$ (the form is called \emph{normalised} then), the denominators
  of all Fourier coefficients are then divisible only by primes less than the
  weight.
\item if the first non-zero Fourier coefficient of the extremal quasimodular
  form is positive, then all Fourier coefficients are positive.
\end{itemize}
The first conjecture has been proved for $r=1$ in
\cite{Pellarin_Nebe2020:extremal_quasi_modular} and
\cite{Mono2020:conjecture_kaneko_koike}. It has been proved in full generality
for $1\leq r\leq4$ in \cite{Grabner2020:quasimodular_forms}.

In the course of the following section we will prove the following theorem,
which partially settles the second conjecture.
\begin{theorem}\label{thm1}
  Let $g_w^{(r)}$ be a normalised extremal quasimodular form of weight $w$ and
  depth $r\leq4$. Then all but possibly finitely many Fourier coefficients are
  positive.
\end{theorem}
\section{Proof of Theorem~\ref{thm1}}
The proof of Theorem~\ref{thm1} will be done separately for the values of the
depth parameter $r$. We will make frequent use of recursive relations for
extremal quasimodular forms derived in
\cite{Grabner2020:quasimodular_forms}. Notice that by definition an extremal
quasimodular form is in $\QEz_w^{(r)}\oplus\QS_w^{(r)}$.

The proofs follow the general scheme
\begin{itemize}
\item express the form $g_w^{(r)}$ in terms of a linear recurrence obtained in
  \cite{Grabner2020:quasimodular_forms}
\item use this recurrence to obtain a linear recurrence for the coefficients of
  $\overline{f_w^{(\ell)}}$ $(\ell=1,\ldots,r)$ in the decomposition of
  $\overline{g_w^{(r)}}$
\item rewrite this decomposition in terms of $\overline{D^\ell E_{w-2\ell}}$
  and observe the positivity of the asymptotic main term originating from
  $DE_{w-2}$.
\end{itemize}
The recursions obtained in \cite[Section~6]{Grabner2020:quasimodular_forms}
contain a positive factor, which ensures that the forms are normalised, which
is important in the context there. In this section we use these recursions
without this factor and at some occasions change this factor, which does not
affect the sign of the coefficients.

\subsection{Depth $1$}\label{sec:depth-1}
Using \cite[Proposition~6.1]{Grabner2020:quasimodular_forms} we define a
sequence of quasimodular forms by
\begin{equation}\label{eq:recurr1}
  \begin{split}
    g_6^{(1)}&=E_2E_4-E_6=-f_6^{(1)}=3DE_4\\
    g_{w+6}^{(1)}&=E_4\partial_{w-1}g_w^{(1)}-\frac{w+1}{12}E_6g_w^{(1)}\\
    g_{w+2}^{(1)}&=\frac{12}{w-1}\partial_{w-1}g_w^{(1)}\\
    g_{w+4}^{(1)}&=E_4g_w^{(1)}
  \end{split}
\end{equation}
for $w\equiv0\pmod6$. These forms are then extremal quasimodular forms of
weight $w$ and depth $1$ with positive coefficient of the first non vanishing
term of its Fourier expansion. 

By the fact that $\QEz_w^1$ is one dimensional we set
\begin{equation*}
  \overline{g_w^{(1)}}=C_w\overline{f_w^{(1)}}.
\end{equation*}
Inserting this into \eqref{eq:recurr1} and using \eqref{eq:Serre-fk} and
\eqref{eq:f_w+l} gives
\begin{align*}
  \overline{g_{w+6}^{(1)}}&=C_w\left(\overline{E_4\partial_{w-1}f_w^{(1)}}
    -\frac{w+1}{12}\overline{E_6f_w^{(1)}}\right)\\
  &=  -C_w\left(\frac{w-1}{12}\overline{E_4f_{w+2}^{(1)}}+
  \frac{w+1}{12}\overline{E_6f_w^{(1)}}\right)=
  -\frac{w}{6}C_w\overline{f_{w+6}^{(1)}},
\end{align*}
from which we derive
\begin{equation*}
  C_{w}=(-1)^{w/6}\left(\frac w6-1\right)!.
\end{equation*}
Together with \eqref{eq:DEw} this gives
\begin{equation*}
  \overline{g_{6k}^{(1)}}=(-1)^{k-1}\frac{6(k-1)!}{3k-1}\overline{DE_{6k-2}}.
\end{equation*}
Applying the third equation in \eqref{eq:recurr1} we obtain
\begin{equation*}
  \overline{g_{6k+2}^{(1)}}=\frac{12}{6k-1}\partial_{6k-1}\overline{g_{6k}^{(1)}}
  =(-1)^k\frac{2(k-1)!}{k}\overline{DE_{6k}}
\end{equation*}
and
\begin{equation*}
  \overline{g_{6k+4}^{(1)}}=\overline{E_4g_{6k}^{(1)}}=
  (-1)^{k-1}\frac{6(k-1)!}{3k+1}\overline{DE_{6k+2}},
\end{equation*}
which gives
\begin{equation}\label{eq:g_w-final}
  \overline{g_w^{(1)}}=(-1)^{\frac w2-1}
  \frac{12\left(\lfloor\frac w6\rfloor-1\right)!}{w-2}\overline{DE_{w-2}}.
\end{equation}

The Fourier coefficients of $g_{w}^{(1)}$ are then given by
\begin{equation*}
  \frac{24\left(\lfloor\frac w6\rfloor-1\right)!}
  {|B_{w-2}|}n\sigma_{w-3}(n)+
  \mathcal{O}\left(n^{\frac{w-1}2}\sigma_0(n)\right).
\end{equation*}
Notice that the first term is of order $n^{w-2}$. Thus we have proved
Theorem~\ref{thm1} for $r=1$.
\subsection{Depth $2$}\label{sec:depth-2}
Using \cite[Proposition~6.2]{Grabner2020:quasimodular_forms} we define a
sequence of quasimodular forms by
\begin{equation}
  \label{eq:gw-2}
  \begin{split}
    g_4^{(2)}&=E_4-E_2^2=-12DE_2=2f_4^{(1)}-f_4^{(2)}\\
    g_{w+4}^{(2)}&=w(w+1)E_4g_w^{(2)}-36\partial_{w-2}^2g_w^{(2)}\\
    g_{w+2}^{(2)}&=\frac{12}{w-2}\partial_{w-2}g_w^{(2)}
  \end{split}
\end{equation}
for $w\equiv0\pmod4$. The form $g_w^{(2)}$ is then an extremal quasimodular
form of weight $w$ and depth $2$ with positive coefficient of the first non
vanishing term in its Fourier expansion.

We make the ansatz
\begin{equation*}
  \overline{g_{4k}^{(2)}}
  =a_{4k}\overline{f_{4k}^{(1)}}+b_{4k}\overline{f_{4k}^{(2)}}
\end{equation*}
with $a_4=2$ and $b_4=-1$. Applying  \eqref{eq:Serre-fk} twice gives
\begin{align*}
  \overline{\partial_{4k-2}^2f_{4k}^{(1)}}&
  =\frac{2k(2k-1)}{36}\overline{f_{4k+4}^{(1)}}+
  \frac{4k-1}{72}\overline{f_{4k+4}^{(2)}}\\
  \overline{\partial_{4k-2}^2f_{4k}^{(2)}}&
  =\frac{2k(2k-1)}{36}\overline{f_{4k+4}^{(2)}}.
\end{align*}
Inserting this into the recurrence \eqref{eq:gw-2} then gives
\begin{equation}\label{eq:matrix-recurr2}
  \begin{pmatrix}
    a_{4k+4}\\b_{4k+4}
  \end{pmatrix}=
    \begin{pmatrix}
      6k(2k+1)&0\\
      -\frac{4k-1}2&6k(2k+1)
    \end{pmatrix}
    \begin{pmatrix}
    a_{4k}\\b_{4k}
  \end{pmatrix}.
\end{equation}
This recurrence has the solutions
\begin{align*}
  a_{4k}&=2\cdot3^{k-1}(2k-1)!\\
  b_{4k}&=-3^k(2k-1)!\left(\frac13+
    \sum_{\ell=1}^{k-1}\frac{4\ell-1}{18\ell(2\ell+1)}\right),
\end{align*}
which can be seen from
\begin{equation*}
  a_{4k+4}=6k(2k+1)a_{4k}=3(2k+1)2k\cdot 2\cdot 3^{k-1}(2k-1)!=2\cdot3^k(2k+1)!
\end{equation*}
and
\begin{align*}
  b_{4k+4}&=6k(2k+1)b_{4k}-\frac{4k-1}{2}a_{4k}\\
  &=  -3^{k+1}(2k+1)!\left(\frac13+
    \sum_{\ell=1}^{k-1}\frac{4\ell-1}{18\ell(2\ell+1)}\right)-
  (4k-1)3^{k-1}(2k-1)!\\
  &= -3^{k+1}(2k+1)!\left(\frac13+
    \sum_{\ell=1}^{k-1}\frac{4\ell-1}{18\ell(2\ell+1)}
    +\frac{4k-1}{18k(2k+1)}\right).
\end{align*}
Similarly we obtain
\begin{align*}
  a_{4k+2}&=-2\cdot3^{k-1}(2k-1)!\\
  b_{4k+2}&=3^k(2k-1)!\left(\frac13+
    \sum_{\ell=1}^{k-1}\frac{4\ell-1}{18\ell(2\ell+1)}-\frac1{3(2k-1)}\right).
\end{align*}
Thus we have
\begin{equation*}
  \overline{g_w^{(2)}}=-\frac{12}{w-2}a_w\overline{DE_{w-2}}
  +\frac{144}{(w-4)(w-3)}b_w\overline{D^2E_{w-4}},
\end{equation*}
where we have used \eqref{eq:DEw} to rewrite $\overline{f_w^{(k)}}$ ($k=1,2$)
in terms of $\overline{DE_{w-2}}$ and $\overline{D^2E_{w-4}}$. Observing that
the sign of $-\frac{a_w}{B_{w-2}}$ is always positive, whereas the sign of
$\frac{b_w}{B_{w-4}}$ is always negative, we derive the $n$-th Fourier
coefficient of $g_w^{(2)}$ using \eqref{eq:DEw} and \eqref{eq:D-eisenstein}
\begin{equation*}
  \frac{24|a_w|}{|B_{w-2}|}n\sigma_{w-3}(n)
  -\frac{288|b_w|}{(w-3)|B_{w-4}|}n^2\sigma_{w-5}(n)
  +\BigOh(n^{\frac{w-1}2}\sigma_0(n)).
\end{equation*}
Notice that the first term is asymptotically dominating and positive, whereas
the second term is negative. This proves Theorem~\ref{thm1} for $r=2$.
\subsection{Depth $3$}
Using \cite[Proposition~6.3]{Grabner2020:quasimodular_forms} we define a
sequence of quasimodular forms by
\begin{equation}
  \label{eq:gw-3}
  \begin{split}
    g_6^{(3)}&=5E_2^3-3E_2E_4-2E_6=
    -12f_6^{(1)}+15f_6^{(2)}-5f_6^{(3)}\\
    g_{w+6}^{(3)}&=48(7w^2+42w+60)\partial_{w-3}^3g_w^{(3)}\\
    &- (15 w^4+96 w^3+ 151 w^2-30
    w-116)E_4\partial_{w-3}g_w^{(3)} \\
    &-\frac16(w+1)(9w^4+45w^3+40w^2+24w+144)E_6g_w^{(3)}\\
    g_{w+2}^{(3)}&=\partial_{w-3}g_w^{(3)}\\
    g_{w+4}^{(3)}&=(w+1)(3w+1)E_4g_w^{(3)}-48\partial_{w-3}^2g_w^{(3)}
  \end{split}
\end{equation}
for $w\equiv0\pmod6$. These forms are then extremal quasimodular forms of
weight $w$ and depth $3$ with positive coefficient of the first non vanishing
term of its Fourier expansion.

We make the ansatz
\begin{equation*}
  \overline{g_{w}^{(3)}}=
  a_{w}\overline{f_{w}^{(1)}}+b_{w}\overline{f_{w}^{(2)}}
  +c_{w}\overline{f_{w}^{(3)}}
\end{equation*}
with $a_6=-12$, $b_6=15$, and $c_6=-5$ and first consider the case $w=6k$.
Applying \eqref{eq:Serre-fk} thrice gives
\begin{align*}
  \overline{\partial_{6k-3}^3f_{6k}^{(1)}}&=
  -\frac{(6k+1)(6k-1)(2k-1)}{576}\overline{f_{6k+6}^{(1)}}
  -\frac{108k^2-36k-1}{864}\overline{f_{6k+6}^{(2)}}\\
  &\quad-  \frac{6k-1}{288}\overline{f_{6k+6}^{(3)}}\\
  \overline{\partial_{6k-3}^3f_{6k}^{(2)}}&=
  -\frac{(6k+1)(6k-1)(2k-1)}{576}\overline{f_{6k+6}^{(2)}}
  -\frac{108k^2-36k-1}{1728}\overline{f_{6k+6}^{(3)}}\\
  \overline{\partial_{6k-3}^3f_{6k}^{(3)}}&=
  -\frac{(6k+1)(6k-1)(2k-1)}{576}\overline{f_{6k+6}^{(3)}}.
\end{align*}

Then a computation similar to the one which gave \eqref{eq:matrix-recurr2}
gives the recurrence
\begin{multline*}
  \begin{pmatrix}
    a_{6k+6}\\b_{6k+6}\\c_{6k+6}
  \end{pmatrix}\\=
\begin{pmatrix}
 \scriptscriptstyle-96 k(2k+1)^2 (3k+1)(3k+2) & 0 & 0 \\
 \scriptscriptstyle8 (2k+1) \left(108 k^3+99 k^2+17 k-2\right) &
 \scriptscriptstyle-96 k(2k+1)^2 (3k+1)(3k+2) & 0 \\
 \scriptscriptstyle-(6k-1)(42k^2+42k+10) &
 \scriptscriptstyle4 (2k+1) \left(108 k^3+99 k^2+17 k-2\right) &
 \scriptscriptstyle-96 k(2k+1)^2 (3k+1)(3k+2) \\
\end{pmatrix}
\begin{pmatrix}
  a_{6k}\\b_{6k}\\c_{6k}
\end{pmatrix}.
\end{multline*}
Notice that this recurrence implies that $(-1)^ka_{6k}$, $(-1)^{k-1}b_{6k}$,
and $(-1)^kc_{6k}$ are positive for all $k\geq1$.
Applying the third equation in \eqref{eq:gw-3} and using \eqref{eq:Serre-fk}
gives
\begin{multline*}
  \overline{g_{6k+2}^{(3)}}=
  -\frac{2k-1}4a_{6k}\overline{f_{6k+2}^{(1)}}-
  \left(\frac{2k-1}4b_{6k}+\frac16a_{6k}\right)\overline{f_{6k+2}^{(2)}}\\
  -\left(\frac{2k-1}4c_{6k}+\frac1{12}b_{6k}\right)\overline{f_{6k+2}^{(3)}};
\end{multline*}
similarly, the fourth equation in \eqref{eq:gw-3} gives
\begin{multline*}
  \overline{g_{6k+4}^{(3)}}=
  \frac{1}{48} \left(5172 k^2+1160 k+47\right)a_{6k}\overline{f_{6k+4}^{(1)}}\\
  +\left(\frac{1}{48} \left(5172 k^2+1160 k+47\right)b_{6k}
    -\frac1{36}a_{6k}\right)\overline{f_{6k+4}^{(2)}}\\+
\left(\frac{1}{48} \left(5172 k^2+1160 k+47\right)c_{6k}
    -\frac1{144}b_{6k}\right)\overline{f_{6k+4}^{(3)}}.
\end{multline*}

Together with \eqref{eq:DEw} and \eqref{eq:D-eisenstein} this gives
\begin{multline*}
  a_w\frac{24}{B_{w-2}}n\sigma_{w-3}(n)
    +b_w\frac{288}{(w-3)B_{w-4}}n^2\sigma_{w-5}(n)\\
    +c_w\frac{3456}{(w-4)(w-5)B_{w-6}}n^3\sigma_{w-7}(n)
  +\BigOh(n^{\frac{w-1}2}\sigma_0(n)).
\end{multline*}
The first term is positive by our discussion of the sign of $a_{6k}$ and the
signs of the Bernoulli numbers. It is of order $n^{w-2}$ and thus dominates the
other terms. This implies the assertion of Theorem~\ref{thm1} for $r=3$.
\subsection{Depth $4$}\label{sec:depth-4}
Using \cite[Proposition~6.4]{Grabner2020:quasimodular_forms} we define a
sequence of quasimodular forms by
\begin{align*}
    g_{12}^{(4)}&=13025 E_4^3-12796 E_6^2+
    3852 E_2E_4E_6-2706 E_2^2 E_4^2\\
    &+27500 E_2^3 E_6-
      28875 E_2^4 E_4\\&=34560f_{12}^{(1)}-93456f_{12}^{(2)}+88000f_{12}^{(3)}
      -28875f_{12}^{(4)}-\frac{15377966208}{691}\Delta\\
      g_{w+12}^{(4)}&=-p_0(w)E_4\partial_{w-4}^4g_w^{(4)}+
    \frac{(w+4)^4}{12}p_1(w)E_6\partial_{w-4}^3g_w^{(4)}\\
    &+\frac1{720}p_2(w) E_4^2\partial_{w-4}^2g_w^{(4)}
    +\frac1{8640}p_3(w)
    E_4E_6\partial_{w-4}g_w^{(4)}\\
    &+\left(\frac{w+1}{25920}p_4(w)E_4^3+
      \frac{(w+1)(w+4)^4}{15}p_5(w)\Delta\right)g_w^{(4)}
\end{align*}
\begin{align*}
    g_{w+2}^{(4)}&=\partial_{w-4}g_w^{(4)}\\
    g_{w+4}^{(4)}&=(w+1)(2w+1)E_4g_w^{(4)}-18\partial_{w-4}^2g_w^{(4)}\\
    g_{w+6}^{(4)}&=\scriptstyle\left(17 w^2+78
      w+90\right)\partial_{w-4}^3g_w^{(4)}
    -    \frac1{144}\left(191 w^4+1008 w^3+1504 w^2+192
      w-576\right)E_4\partial_{w-4}g_w^{(4)}\\
    &\scriptstyle-\frac1{432}(w+1) \left(81 w^4+376 w^3+560 w^2+528
      w+576\right)E_6g_w^{(4)}
\end{align*}
\begin{align*}
    g_{w+8}^{(4)}&=\scriptstyle-\left(1313 w^6+28678 w^5+255122 w^4+1183008 w^3
      +3016512
      w^2+ 4012416 w+2177280\right)\partial_{w-4}^4g_w^{(4)}\\
    &\scriptstyle
    +\frac1{144}\bigl(13423 w^8+295800 w^7+2645368 w^6+12166080 w^5+29311504
      w^4+29020416
      w^3-15653376 w^2\\
      &\quad\scriptstyle-56692224 w-33094656\bigr)E_4\partial_{w-4}^2g_w^{(4)}\\
    &\scriptstyle+\frac1{432}\bigl(6561 w^9+136994 w^8+1139536
    w^7+4759344 w^6+10294016 w^5+11541472 w^4+14671104 w^3
    \\&\quad\scriptstyle+41398272
      w^2+63016704
      w+31974912\bigr)E_6\partial_{w-4}g_w^{(4)}\\
    &\scriptstyle+\frac1{2592}(w+1) \bigl(2048 w^9+38685 w^8+287792
    w^7+1130616 w^6+3110288 w^5+8497968 w^4\\
    &\quad\scriptstyle+18484992 w^3+14141952w^2-20570112 w-30855168\bigr)
    E_4^2g_w^{(4)}\\
    g_{w+10}&=\scriptstyle\left(293 w^4+4332 w^3+22968 w^2+51192 w+40824\right)E_4
    \partial_{w-4}^3g_w^{(4)}\\
    &\scriptstyle-\frac43\left(w^5+15 w^4+90 w^3+270 w^2+405 w+243\right)
    E_6\partial_{w-4}^2g_w^{(4)}\\
    &\scriptstyle-\frac1{144}\left(3311 w^6+51234 w^5+291550 w^4+731040
      w^3+717696 w^2-2592 w-256608\right)E_4^2\partial_{w-4}g_w^{(4)}\\
    &\scriptstyle-\frac1{432}(w+1) \left(1313 w^6+19430 w^5+104354
      w^4+251616 w^3+310464 w^2+300672 w+248832\right)E_4E_6g_w^{(4)}
\end{align*}
for $w\equiv0\pmod{12}$. These forms are then extremal quasimodular forms of
weight $w$ and depth $4$ with positive coefficient of the first non vanishing
term of its Fourier expansion. The polynomials $p_0,\ldots,p_5$ are given by
\begin{align*}
    p_0(w)&=\scriptstyle 53567 w^{14}+4499628 w^{13}+173318340
    w^{12}+4055616864 w^{11}+
    64374205218 w^{10}\notag\\
    &\scriptstyle+732790207224 w^9+6165100658404 w^8+38914973459904 w^7+
    185044363180416 w^6\notag\\
    &\scriptstyle
    +659055640624128 w^5+1729058937394176 w^4+
    3237068849283072 w^3\notag\\
    &\scriptstyle
    +4084118362128384 w^2+3105388005949440 w+1072718335180800\notag\\
\end{align*}
\begin{align*}
    p_1(w)&=\scriptstyle21257 w^{11}+1465884
      w^{10}+45186990 w^9+821051740 w^8+9759703548 w^7\notag\\
    &\scriptstyle +
      79588527156 w^6
      +453687847200 w^5
      +1804779218520 w^4+4900200364800 w^3\notag\\
    &\scriptstyle+
      8628400143360 w^2
      +8845395333120 w+3990767616000\notag\\
\end{align*}
\begin{align*}
p_2(w)&=\scriptstyle2662740 w^{16}+224120550
    w^{15}+8648003840 w^{14}+202621853220 w^{13}\notag\\
    &\scriptstyle+
  3217542322665 w^{12}
  +36586266504480 w^{11}+306658234963680 w^{10}+\notag\\
    &\scriptstyle
  1919356528986240 w^9
  +8970889439482816 w^8+30866477857195008 w^7\notag\\
    &\scriptstyle+
  75319919247624192 w^6
  +118664936756305920 w^5+83296021547483136 w^4\notag\\
    &\scriptstyle
  -82769401579438080 w^3
  -258790551639293952 w^2-245119018746249216 w \notag\\
    &\scriptstyle -86822757140004864\notag\\
\end{align*}
\begin{align*}
p_3(w)&=\scriptstyle4272785 w^{17}+351970350
    w^{16}+13234823080 w^{15}+300533087760 w^{14}\notag\\
    &\scriptstyle+4592608729932 w^{13}
    +49787752253076 w^{12}+392868254956864 w^{11}\notag\\
    &\scriptstyle  +2274866661846720 w^{10}+9597118952486912 w^9
    +28789901067644544 w^8 \notag\\
    &\scriptstyle +58741997991303168 w^7
  +79017091035181056 w^6
  +100071999240486912w^5\notag\\
    &\scriptstyle+278562611915587584 w^4+779359222970449920 w^3
  +1260737947219525632 w^2\notag\\
    &\scriptstyle+1054463073573666816 w+355736061701259264\notag\\
\end{align*}
\begin{align*}
p_4(w)&=\scriptstyle517135 w^{17}+40772970 w^{16}
  +1455719580 w^{15}+31076826800 w^{14}+441034824168 w^{13}\notag\\
    &\scriptstyle
  +4375275488634 w^{12}+31084796008256 w^{11}  +160090786631040 w^{10}+608772267089664 w^9\notag\\
    &\scriptstyle
  +1834128793979392 w^8 +5229385586024448 w^7+15775977503047680 w^6\notag\\
    &\scriptstyle
+40287913631023104  w^5+57115900062203904 w^4-19258645489385472 w^3\notag\\
    &\scriptstyle
  -224285038806564864 w^2  -343616934723452928 w-182090547421249536\notag\\
\end{align*}
\begin{align*}
p_5(w)&=\scriptstyle531441 w^{13}+36690686
    w^{12}+1133566168 w^{11}+20680195920 w^{10}+247548700336 w^9\notag\\
    &\scriptstyle
    +2043291298652 w^8+11897624359104 w^7+49185666453888 w^6
    +143692776009216 w^5\notag\\
   &\scriptstyle
   +293687697411072 w^4+418695721574400 w^3+426532499288064 w^2\notag\\
   &\scriptstyle
+316421756411904 w+135523565862912.\notag
\end{align*}

As before we make the ansatz
\begin{equation*}
  \overline{g_{w}^{(4)}}=a_{w}\overline{f_{w}^{(1)}}
  +b_{w}\overline{f_{w}^{(2)}}+c_{w}\overline{f_w^{(3)}}
  +d_w\overline{f_w^{(4)}},
\end{equation*}
which gives a recurrence
\begin{equation*}
  \begin{pmatrix}
    a_{12(k+1)}\\b_{12(k+1)}\\c_{12(k+1)}\\d_{12(k+1)}\\
  \end{pmatrix}=
  \begin{pmatrix}
    \lambda_k&0&0&0\\
    -*&\lambda_k&0&0\\
    +*&-*&\lambda_k&0\\
    -*&+*&-*&\lambda_k
  \end{pmatrix}
  \begin{pmatrix}
    a_{12k}\\b_{12k}\\c_{12k}\\d_{12k}\\
  \end{pmatrix},
\end{equation*}
where $\lambda_k$ is a polynomial of degree $18$, which factors into rational
linear factors and $\pm*$ denotes positive/negative entries. This together with
the signs of the initial values $a_{12}=34560$, $b_{12}=-93456$,
$c_{12}=88000$, and $d_{12}=-28875$ shows that $a_{12k}$, $-b_{12k}$,
$c_{12k}$, and $-d_{12k}$ are all positive. From this it follows that
$(-1)^{\lfloor\frac w2\rfloor}a_w$ is positive. Finally, this gives the
asymptotic formula
\begin{multline*}
  \frac{24a_w}{B_{w-2}}n\sigma_{w-3}(n)
  +\frac{288b_w}{(w-3)B_{w-4}}n^2\sigma_{w-5}(n)\\+
  \frac{3456c_w}{(w-4)(w-5)B_{w-6}}n^3\sigma_{w-7}(n)\\
  -\frac{41472d_w}{(w-5)(w-6)(w-7)B_{w-8}}n^4\sigma_{w-9}(n)+
  \BigOh\left(n^{\frac{w-1}2}\sigma_0(n)\right)
\end{multline*}
for the Fourier coefficients of $g_w^{(4)}$, where we have used \eqref{eq:DEw}
and \eqref{eq:D-eisenstein} for the explicit expression of the terms coming
from $f_w^{(k)}$ ($k=1,\ldots,4$). The first term asymptotically dominates and
is positive by our discussion of the sign of $a_w$ and the sign of the
Bernoulli number. This implies the theorem for $r=4$.
\section{Numerical experiments}\label{sec:numer-exper}
In \cite{Jenkins_Rouse2011:bounds_coefficients_cusp} an explicit bound for the
Fourier coefficients of cusp forms has been derived.
\begin{theorem}[Theorem~1 in \cite{Jenkins_Rouse2011:bounds_coefficients_cusp}]
  \label{thm-jenkins-rouse}
  Let
  \begin{equation*}
    G(z)=\sum_{n=1}^\infty g(n)q^n
  \end{equation*}
  be a cusp form of weight $w$. Then
  \begin{equation}
    \label{eq:jenkins-rouse}
    \begin{split}
      |g(n)|&\leq\sqrt{\log w}\Biggl(
        11\sqrt{\sum_{m=1}^\ell\frac{|g(m)|^2}{m^{w-1}}}\\
      &+
        \frac{e^{18.72}(41.41)^{w/2}}{w^{(w-1)/2}} \left|\sum_{m=1}^\ell
          g(m)e^{-7.288m}\right|\Biggr)n^{\frac{w-1}2}\sigma_0(n),
    \end{split}
\end{equation}
where $\ell$ is the dimension of the space of cusp forms of weight $w$.
\end{theorem}

For an application of this theorem we write an extremal quasimodular form of
depth $r$ as
\begin{equation}\label{eq:gwr-decomp}
  g_w^{(r)}=\sum_{\ell=1}^r c_\ell D^\ell E_{w-2\ell}+
  \sum_{\ell=0}^rD^\ell\alpha_{w-2\ell},
\end{equation}
where $c_1,\ldots,c_r$ are the coefficients computed in
Sections~\ref{sec:depth-1} to~\ref{sec:depth-4} for the according values of
$r$, and $\alpha_{w-2r},\ldots,\alpha_w$ are cusp forms of weights
$w-2r,\ldots,w$. Rewriting the forms $g_w^{(r)}$ is done using the expressions
for derivatives given in \eqref{eq:D-serre}. In order to make this more clear,
we give the according conversion formula for the case $r=1,2$
\begin{align*}
  &A_w+E_2B_{w-2}=\left(A_w-\frac{12}{w-2}\partial_{w-2}B_{w-2}\right)+
  D\left(\frac{12}{w-2}B_{w-2}\right)\\
  &A_w+B_{w-2}E_2+C_{w-4}E_2^2\\
  &=\left(A_w-\frac{12}{w-2}\partial_{w-2}B_{w-2}
    +\frac{144}{(w-2)(w-3)}\partial_{w-4}^2C_{w-4}+\frac1{w-3}E_4C_{w-4}\right)\\
  &+D\left(\frac{12}{w-2}B_{w-2}
    -\frac{288}{(w-2)(w-4)}\partial_{w-4}C_{w-4}\right)\\
&+D^2\left(\frac{144}{(w-3)(w-4)}C_{w-4}\right).
\end{align*}
The cases $r=3,4$ are much more complex; the computations were done using
\texttt{Mathematica}. The \texttt{Mathematica} source code is available at
\cite{Grabner2020:mathematica_files}.

Theorem~\ref{thm-jenkins-rouse} can then be applied to the forms
$\alpha_{w-2\ell}$ ($\ell=0,\ldots,r$) to derive bounds of the form $C_\ell
n^{\frac{w-1}2}\sigma_0(n)$ for the Fourier coefficients of the forms
$D^\ell\alpha_{w-2\ell}$. This gives the bound
$(C_0+\cdots+C_r)n^{\frac{w-1}2}\sigma_0(n)$ for the Fourier coefficient of the
second sum in \eqref{eq:gwr-decomp}.

The Fourier coefficients of the terms in
the first sum are
\begin{equation*}
  c_\ell \frac{2(w-2\ell)}{B_{w-2\ell}}n^\ell\sigma_{w-2\ell-1}(n).
\end{equation*}
For these we use the bounds
\begin{equation*}
  n^{w-2\ell-1}\leq\sigma_{w-2\ell-1}(n)\leq
  n^{w-2\ell-1}\sum_{d\mid n}d^{2\ell+1-w}\leq\zeta(w-2\ell-1)n^{w-2\ell-1}
\end{equation*}
and  $\sigma_0(n)\leq2\sqrt n$
to derive an explicit lower bound for the Fourier coefficients of
$g_w^{(r)}$. This bound is positive for $n\geq N_0$ for an explicitly
computable value $N_0$.

For the remaining finitely many Fourier coefficients
positivity can be checked with the help of a computer. We have performed these
computations for $1\leq r\leq4$ and $w\leq200$.

\begin{acknowledgement}
  The author is grateful to an anonymous referee for the many valuable comments
  that improved the readability of the paper.
\end{acknowledgement}

\begin{thebibliography}{10}

\bibitem{Berndt_Knopp2008:heckes_theory_modular}
B.~C. Berndt and M.~I. Knopp, \emph{{H}ecke's theory of modular forms and
  {D}irichlet series}, Monographs in Number Theory, vol.~5, World Scientific
  Publishing Co. Pte. Ltd., Hackensack, NJ, 2008.

\bibitem{Bruinier_Geer_Harder+2008:1_2_3-modular}
J.~H. Bruinier, G.~van~der Geer, G.~Harder, and D.~Zagier, \emph{{T}he 1-2-3 of
  modular forms}, Universitext, Springer-Verlag, Berlin, 2008, Lectures from
  the Summer School on Modular Forms and their Applications held in
  Nordfjordeid, June 2004, Edited by Kristian Ranestad.

\bibitem{Choie_Lee2019:jacobi_like_forms}
Y.~J. Choie and M.~H. Lee, \emph{{J}acobi-{L}ike {F}orms, {P}seudodifferential
  {O}perators, and {Q}uasimodular {F}orms}, Monographs in Mathematics, Springer
  International Publishing, 2019.

\bibitem{Deligne1974:la_conjecture_weil}
P.~Deligne, \emph{{L}a conjecture de {W}eil: {I}}, Publ. Math. I.H.E.S.
  \textbf{43} (1974), 273--307.

\bibitem{Diamond_Shurman2005:first_course_modular}
F.~Diamond and J.~Shurman, \emph{A first course in modular forms}, Graduate
  Texts in Mathematics, vol. 228, Springer-Verlag, New York, 2005.

\bibitem{Grabner2020:quasimodular_forms}
P.~J. Grabner, \emph{{Q}uasimodular forms as solutions of modular differential
  equations}, Int. J. Number Theory \textbf{16} (2020), 2233–2274.

\bibitem{Grabner2020:mathematica_files}
\bysame, \emph{\texttt{{M}athematica}-files},
  https://doi.org/10.5281/zenodo.4153177, 2020.

\bibitem{Iwaniec1997:topics_classical_automorphic}
H.~Iwaniec, \emph{{T}opics in classical automorphic forms}, Graduate Studies in
  Mathematics, vol.~17, American Mathematical Society, Providence, RI, 1997.

\bibitem{Iwaniec_Kowalski2004:analytic_number_theory}
H.~Iwaniec and E.~Kowalski, \emph{{A}nalytic number theory}, vol.~53, American
  Mathematical Society, Providence, RI, 2004.

\bibitem{Jenkins_Rouse2011:bounds_coefficients_cusp}
P.~Jenkins and J.~Rouse, \emph{Bounds for coefficients of cusp forms and
  extremal lattices}, Bull. Lond. Math. Soc. \textbf{43} (2011), no.~5,
  927--938.

\bibitem{Kaneko_Koike2006:extremal_quasimodular_forms}
M.~Kaneko and M.~Koike, \emph{{O}n extremal quasimodular forms}, Kyushu J.
  Math. \textbf{60} (2006), no.~2, 457--470.

\bibitem{Kaneko_Zagier1995:generalized_jacobi_theta}
M.~Kaneko and D.~Zagier, \emph{{A} generalized {J}acobi theta function and
  quasimodular forms}, The moduli space of curves ({T}exel {I}sland, 1994),
  Progr. Math., vol. 129, Birkh\"{a}user Boston, Boston, MA, 1995,
  pp.~165--172.

\bibitem{Lang1995:introduction_to_modular}
S.~Lang, \emph{{I}ntroduction to modular forms}, Grundlehren der Mathematischen
  Wissenschaften [Fundamental Principles of Mathematical Sciences], vol. 222,
  Springer-Verlag, Berlin, 1995, With appendixes by D. Zagier and Walter Feit,
  Corrected reprint of the 1976 original.

\bibitem{Mono2020:conjecture_kaneko_koike}
A.~Mono, \emph{{O}n a conjecture of {K}aneko and {K}oike},
  \url{https://arxiv.org/abs/2005.06882v1}, May 2020.

\bibitem{Pellarin_Nebe2020:extremal_quasi_modular}
F.~Pellarin, \emph{{O}n extremal quasi-modular forms after {K}aneko and
  {K}oike}, Kyushu J. Math. \textbf{74} (2020), with an appendix by G.~Nebe, to
  appear, \url{https://arxiv.org/abs/1910.11668}.

\bibitem{Royer2012:quasimodular_forms_introduction}
E.~Royer, \emph{{Q}uasimodular forms: an introduction}, Ann. Math. Blaise
  Pascal \textbf{19} (2012), no.~2, 297--306.

\bibitem{Shimura2012:modular_forms_basics}
G.~Shimura, \emph{Modular forms: basics and beyond}, Springer Monographs in
  Mathematics, Springer, New York, 2012.

\bibitem{Stein2007:modular_forms_computational}
W.~Stein, \emph{Modular forms, a computational approach}, Graduate Studies in
  Mathematics, vol.~79, American Mathematical Society, Providence, RI, 2007,
  With an appendix by Paul E. Gunnells.

\bibitem{Zagier2008:elliptic_modular_forms}
D.~Zagier, \emph{{E}lliptic modular forms and their applications}, in \emph{The
  1-2-3 of modular forms} \cite{Bruinier_Geer_Harder+2008:1_2_3-modular},
  pp.~1--103.

\end{thebibliography}
\providecommand{\bysame}{\leavevmode\hbox to3em{\hrulefill}\thinspace}
\providecommand{\MR}{\relax\ifhmode\unskip\space\fi MR }
\providecommand{\MRhref}[2]{%
  \href{http://www.ams.org/mathscinet-getitem?mr=#1}{#2}
}
\providecommand{\href}[2]{#2}

\end{document}